\documentclass[12pt, reqno]{amsart}
\usepackage{amsmath, amsthm, amscd, amsfonts, amssymb, graphicx, color}
\graphicspath{{figure/}}
\usepackage{color}

\usepackage{mathtools}
\usepackage[percent]{overpic}
\usepackage{amsmath,amsthm,amssymb,amsfonts}
\usepackage[colorlinks=true,linkcolor=blue,citecolor=blue,urlcolor=blue,]{hyperref}
\usepackage{mathrsfs}
\usepackage{tikz-cd}

\newtheorem{theorem}{Theorem}[section]
\newtheorem{lemma}[theorem]{Lemma}
\newtheorem{proposition}[theorem]{Proposition}
\newtheorem{corollary}[theorem]{Corollary}
\theoremstyle{definition}
\newtheorem{definition}[theorem]{Definition}

\theoremstyle{remark}

\numberwithin{equation}{section}

\begin{document}
\setcounter{page}{1}

%------------------------------------------------------------------------------

%Title of the paper
\title[On the isotopy classes of embeddings of surfaces in 5-manifolds]{On the isotopy classes of embeddings of surfaces in 5-manifolds}

%Author names and affiliations
\author[Ruoyu Qiao]{Ruoyu Qiao}

\address{Department of Mathematics, University of Maryland, College Park, USA.}
\email{\textcolor[rgb]{0.00,0.00,0.84}{wmxq@umd.edu}}

\begin{abstract}
Let $f,f':\Sigma\hookrightarrow N$ be two homotopic smooth embeddings of a closed surface in a closed oriented 5-manifold. We show that if $f$  admits a common algebraic dual 3-sphere, or if $\pi_1(N)$ is trivial, then $f$ and $f'$ must be isotopic. This generalizes a result of Kosanovi\'c--Schneiderman--Teichner \cite{Kosanovic2023}.
The proof is based on the construction of an invariant that classifies the isotopy classes of smooth embeddings of surfaces in ambient 5-manifolds within a homotopy class, which may be of independent interest. The invariant is defined in terms of the homotopy groups of the 5-manifold. 
\end{abstract} \maketitle

\section{Introduction and Results}
The relationship between homotopy and isotopy is a classical problem in topology. In 2017,  Gabai \cite{Gabai2017The4L}, \cite{gabai2020} established the 4-dimensional light bulb theorem, which states that if two embedded spheres in an ambient 4-manifold admit a common geometric dual sphere, and if the fundamental group of the ambient manifold has no 2-torsion, then embedded spheres are isotopic. More recently, Kosanovi\'c, Schneiderman, and Teichner investigated analogous questions for embeddings of 2-spheres in 5-manifolds \cite{Kosanovic2023}, and they established a complete classification of all the isotopy classes within a homotopy class. 
\par
Gabai \cite{Gabai2017The4L} also extended the light bulb theorem to ``$G$-inessential'' embeddings of surfaces in 4-manifolds (where $G$ denotes the geometric dual sphere). In contrast, Lin, Wu, Xie, and Zhang \cite{lin2025} proved a negative result in the absence of this condition, showing that there exist infinitely many homotopic embeddings of surfaces in a 4-manifold that admit a common geometric dual sphere but are mutually non-isotopic .\par
In this paper, we consider the ``homotopy versus isotopy'' question for closed oriented 2-manifolds  in 5-manifolds. All the discussions in this paper are in the smooth category. We will construct an invariant that classifies the isotopy classes within a homotopy class by constructing a self-intersection invariant of homotopies, which takes value in a quotient of the free $\mathbb{Z}$ module generated by a (set-theoretic) quotient of the fundamental group of the ambient manifold. \par
%In particular, we will:\\
%(1) explain how the isotopy classification can be described by the generalized self-intersection invariant of homotopies, with Whitney move and Hudson's theorem.\\
%(2) show how we can eliminate the ambiguity given by self-homotopy by analyzing the movements of skeletons.\\
%(3) use the invariant to show some properties of the isotopic classes.\par
Now we give a brief description of the main results. Let $\Sigma$ be a closed oriented surface of positive genus $g$, let $N$ be a 5-dimensional closed oriented manifold. Consider the canonical map 
\begin{equation}
	\label{eqn_pi0_emb_to_map}
p:\pi_0\textup{Emb}(\Sigma,N)\to\pi_0\textup{Map}(\Sigma,N).
\end{equation}
Here, $\textup{Emb}(\Sigma,N)$ denotes the space of embeddings of $\Sigma$ in $N$, and $\textup{Map}(\Sigma,N)$ denotes the space of smooth maps from $\Sigma$ to $N$. 
We will consider the ``homotopy classes of homotopies'' between embeddings.
For an embedding $f:\Sigma\hookrightarrow N^5$, we will use $\mathscr{H}^f_0$ to denote the homotopy classes of self-homotopies of $f$, and use $\mathscr{H}^{f}$ to denote the homotopy classes of homotopies that start with $f$ (see Definitions \ref{def2.1}, \ref{def3.1}).  Let $G$ be the image of $f_*:\pi_1\Sigma\rightarrow\pi_1N$. We will define the set $\mathbb{A}_f$ as a certain quotient of the free $\mathbb{Z}$-module generated by $G\backslash\pi_1N/G$ (see Definitions \ref{def2.2}, \ref{def3.21}), where $G\backslash\pi_1N/G$ denotes the set of double cosets $\{GgG|g\in\pi_1N\}$. The following result is the main theorem of the paper. Here, the map $p$ is given by \eqref{eqn_pi0_emb_to_map}.

\begin{theorem}\label{thm1.1}
    There exists the following commutative diagram:
\begin{center}
\begin{tikzcd}
& \mathscr{H}^f \arrow[r, "\mu"] \arrow[d,""]
& \mathbb{A}_{[f]} \arrow[d, ""] \\
p^{-1}([f]) \arrow[r, "\cong"] & \mathscr{H}^f/\mathscr{H}^f_0 \arrow[r, "\overline\mu"]
& \mathbb{A}_{f}
\end{tikzcd}
\end{center}
where $\mu$ is the intersection number (Definition \ref{def2.5}), $\mathscr{H}^f_0$ acts on $\mathscr{H}^f$ and $\mathbb{A}_{[f]}$ inducing the map $\overline\mu$ (Proposition \ref{prop3.25}), which is a bijection. And there is a bijection $p^{-1}([f])\rightarrow\mathscr{H}^f/\mathscr{H}^f_0$ and $\mathbb{A}_{f}$ (Theorem \ref{thm3.8}).
And the induced map (Definition \ref{def4.1})
$$\textup{fq}:p^{-1}([f])\rightarrow \mathbb{A}_f$$
is a bijection.
\end{theorem}

Theorem \ref{thm1.1} has the following immediate corollary.

\begin{corollary}
	\label{cor_detect_isotopy}
	Suppose $f':\Sigma\hookrightarrow N$ is homotopic to $f$. Then $f$ and $f'$ are isotopic if and only if $\textup{fq}(f')=0\in\mathbb{A}^f$.
\end{corollary}

Note that Theorem \ref{thm1.1} and Corollary \ref{cor_detect_isotopy} do not require any assumptions on the existence of dual spheres. In Section \ref{sec_application}, we will show that if $f$ admits an algebraic dual sphere or if $\pi_1(N)$ is trivial, then $\mathbb{A}_f$ is trivial, and hence every embedding that is homotopic to $f$ must be isotopic to $f$. Here, we say that $f$ admits an algebraic dual sphere if there exists an element in the image of $\pi_3(N)\to H_3(N)$ that has algebraic intersection number $1$ with the fundamental class of $f$. The result is summarized in the following corollary.
\begin{corollary}\label{Coro1.3}
    Let $\Sigma$ be a compact oriented 2-surface and $N$ be a compact oriented 5-manifold. Let $f:\Sigma\rightarrow N$ be an embedding. Then if $N$ is simply connected, or if $f$ admits an algebraic dual sphere, or if the homomorphism $f_*:\pi_1\Sigma\rightarrow\pi_1N$ is surjective, then every embedding $f'$ that is homotopic to $f$ is actually isotopic to $f$.
    \label{corollary3}
\end{corollary}

In contrast, the next corollary shows that there are examples $f:\Sigma\hookrightarrow N$ such that there exist infinitely many embeddings $f':\Sigma\hookrightarrow N$ that are homotopic but not isotopic to $f$.
\begin{corollary}\label{Coro1.4}
    Let $\Sigma$ be a compact oriented 2-surface and $N$ be a compact oriented 5-manifold. Let $f:\Sigma\rightarrow N$ be an embedding and $G\coloneqq f_*(\pi_1\Sigma)\in\pi_1N$. If $\pi_2N,\pi_3N$ are all trivial and there are infinitely many conjugacy classes in the set of cosets $G\backslash\pi_1N/G$ under the conjugate action of $G$, then there exist infinitely many embeddings $f':\Sigma\rightarrow N$ which are homotopic to $f$ but not isotopic to $f$.
\end{corollary}

{\bf Acknowledgement.} The author would like to express sincere gratitude to Prof. Boyu Zhang for constructive suggestions and discussions. The author also thanks Jianning Fu for important advice.

\section{Intersection Number of Homotopies up to Homotopy}

In this section, we will define a homotopy invariant of homotopies to show when a homotopy is homotopic to an isotopy. The main idea comes from the Wall's intersection number \cite{Wall1966}, which is defined to be an obstruction to a given immersion $f:S^n
\rightarrow N^{2n}$ regularly homotopic to an embedding. The obstruction is complete when $n\geq3$, which is exactly our case. Concretely, the invariant relates the intersections of the immersion with loops in $\pi_1N$, generated by the two arcs from the basepoint of $N$ to the intersection point along the two spheres or the two sheets of a sphere. Here, we will use a conjugacy version of the intersection number since the immersed manifolds may not be simply-connected. Later in Definition \ref{def2.5}, we will give a more detailed definition.\par

We start with several definitions:\par
\begin{definition}\label{def2.1}
    Let $f,f':\Sigma\hookrightarrow N$ be two embeddings from a compact oriented 2-surface $\Sigma$ to a compact, oriented 5-manifold $N$. If there is a homotopy $H:\Sigma\times I\rightarrow N$ such that $H_0=f$, $H_1=f'$, we call $H$ a homotopy connecting $f$ and $f'$, or from $f$ to $f'$. If we regard the homotopy $H$ as a map $H:\Sigma\times I\rightarrow N\times I$ that fixes the $I$ coordinate, then we say two such homotopies $H$ and $H'$ are equivalent if and only if $H_0=H'_0$, $H_1$ is isotopic to $H'_1$, and there exists a homotopy $J:\Sigma\times I\times I\rightarrow\Sigma\times I$, such that $J|_{\Sigma\times I\times \{0\}}=H$ and $J|_{\Sigma\times I\times\{1\}}=H'$ and $J|_{\Sigma\times \{0\}\times I}=f$. Denote by $[H]$ the equivalence class of a homotopy $H$. We call a homotopy $H:\Sigma\times I\rightarrow N$ a homotopy of an embedding $f:\Sigma\hookrightarrow N$ if $H_0=f$. Denote by $\mathscr{H}^f$ the space of the homotopy equivalence classes of the homotopies of $f$.
\end{definition}

\begin{definition}\label{def2.2}
    Let $N$ be a compact oriented 5-manifold. Let $\Sigma$ be a compact oriented 2-surface. If $F$ is an immersion $\Sigma\times I\looparrowright N\times I$ and $G\coloneq F_*(\pi_1(\Sigma\times I))\in\pi_1(N\times I)$. Define the group $\mathbb{A}_{[F]}$ as the quotient $\mathbb{Z}G\backslash\pi_1(N\times I)/G\bigg/{\langle\overline{1},\overline{g}+\overline{g}^{-1}\rangle}$, where $\mathbb{Z}G\backslash\pi_1(N\times I)/G$ is the $\mathbb{Z}$ module of the double coset $G\backslash\pi_1(N\times I)/G$. Since $\Sigma\times {0}$ is a deformation retract of $\Sigma$ and $N\times \{0\}$ is a deformation retract of $N\times I$, let $f=F|_{\Sigma\times\{0\}}$, we also denote $\mathbb{A}_{[F]}$ by $\mathbb{A}_{[f]}$.
\end{definition}

In this section, we will establish a map $\mu$ from $\mathscr{H}^f$ to $\mathbb{A}_{[f]}$ and show that this intersection number is also an obstruction.

Thus, the main result of this section is:

\begin{theorem}\label{thm2.3}
    Consider the map $$\mu:\mathscr{H}^f\rightarrow\mathbb{A}_{[f]},$$
    in Definition \ref{def2.5}. A homotopy $H:\Sigma\times I \rightarrow N\times I$ of $f$ is homotopic relative to boundary to an isotopy if and only if the intersection number $\mu([H])\in \mathbb{A}_{[H]}=\mathbb{A}_{[f]}$ vanishes.
\end{theorem}

There are two main differences between our intersection number here and the original version of the Wall's intersection invariant. Firstly, the Wall's intersection number is a regular homotopy invariant, but what we want here is a homotopy invariant. And secondly, the Wall's invariant is valid for embeddings of simply-connected manifolds, but $\pi_1(\Sigma)$ may not be trivial.\par
For the first question, since we only consider the embeddings of 3-manifolds to 6-manifolds, the only difference between homotopy and regular homotopy for homotopy between two immersions is the cusp homotopy, see Figure \ref{fig:figure6}.

\begin{figure}
    \begin{overpic}[scale=0.7]{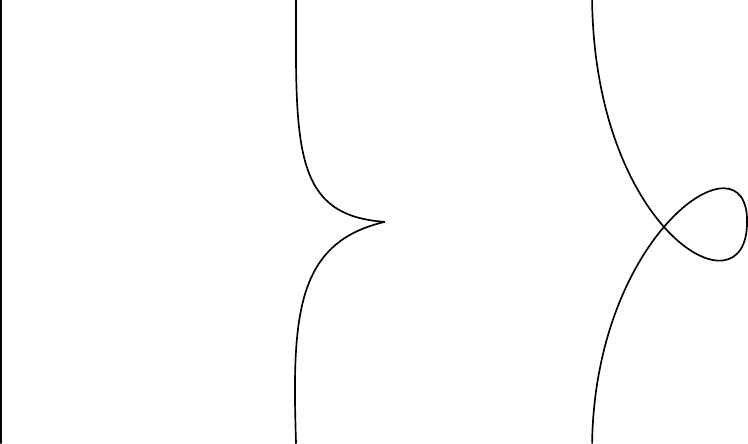}
	\end{overpic}
    \caption{A cusp homotopy}
    \label{fig:figure6}
\end{figure}

More concretely, we have the following lemma:
\begin{lemma}
    Let $\Sigma$ be a compact oriented 2-surface. Let $N$ be an oriented compact 5-manifold. Denote by $q:\pi_0\textup{Imm}(\Sigma\times I,N\times I)\rightarrow\pi_0\textup{Map}(\Sigma\times I,N\times I)$ the projection map from the regular homotopy class to the homotopy classes of immersions fixing the boundary. Then for a given immersion $F:\Sigma\times I\rightarrow N\times I$, the regular homotopy classes $q^{-1}([F])\cong\mathbb{Z}$ which are given by cusp homotopies on the immersion $F$.
\end{lemma}

As this point is not central to the paper, we only sketch the proof.

\begin{proof}[Sketch of proof]
    By Smale-Hirsch theorem \cite{Hirsh1959,Smale1959}, there exists a weak homotopy equivalence given by differential:
\begin{align*}
    D:\textup{Imm}(\Sigma\times I,N\times I) & \rightarrow\textup{Mono}(T(\Sigma\times I),T(N\times I))\\
    F & \mapsto DF:T(\Sigma\times I)\rightarrow T(N\times I),
\end{align*}
where $\textup{Mono}(T(\Sigma\times I),T(N\times I))$ refers to the space of bundle monomorphisms. Thus $q^{-1}([F])$, all the immersions in one homotopy class, is weakly homotopic equivalent to $\textup{Mono}_F(T(\Sigma\times I),T(N\times I))$, the space of bundle monomorphisms whose base space mapping is homotopic to $f$. We may apply the homotopy on the bundle monomorphisms and think of this as a fiber bundle $$V_{3,6}\rightarrow\nu_F(T(\Sigma\times I),T(N\times I))\rightarrow \Sigma\times I.$$
Here $\nu_F(T(\Sigma\times I),T(N\times I))$ is the space of pairs $(x,u)$, where $x\in\Sigma\times I$ and $u:T_x(\Sigma\times I)\rightarrow T_{F(x)}(N\times I)$ is a linear monomorphism, and $V_{3,6}$ is the Stiefel manifold. And the space $\textup{Mono}_F(T(\Sigma\times I),T(N\times I))$ is homotopic equivalent to the space of sections of this fiber bundle, which we denote by $\Gamma\nu_F(T(\Sigma\times I),T(N\times I))$. So we need to compute $\pi_0\Gamma(\nu_i(TM,TN')).$\par

It is well known that $\pi_1(V_{3,6})=\pi_2(V_{3,6})=0$ and $\pi_3(V_{3,6})=\mathbb{Z}$. Finally, consider the CW decomposition $\Sigma=e^0\cup\Sigma^1\cup e^2$. If two sections agree on the boundary of $e^0\times I$, it can be shown that the obstruction to the two sections being homotopic is given by an element in $\pi_1$, see \cite{Hatcher2002}. But since $V_{3,6}$ is simply connected, there's no such obstruction, meaning that any sections agreeing on the boundary are actually homotopic. The same discussion is valid for the 1-skeleton $\Sigma^1$. And for the only 2-cell $e^2$, the obstruction is given by $\pi_3(V_{3,6})\cong\mathbb{Z}$, and it is realized by a single cusp homotopy.
\end{proof}

Now we give the following definition of the intersection number.

\begin{definition}\label{def2.5}
    Let $N$ be a compact, oriented 5-manifold. Let $\Sigma$ be a compact oriented 2-surface. Then, let $F$ be a generic immersion $\Sigma\times I\rightarrow N\times I$ with a whisker $\nu$. Define $G$ and $\mathbb{A}_{[F]}$ as in Definition \ref{def2.2}. Then define the following sum: $$\mu(F) \coloneqq\Sigma_{p\in F\pitchfork F}\epsilon(p)\overline{\alpha}(p)/\thicksim\in\mathbb{A}_{[F]}$$
    where
    \begin{enumerate}
        \item[(a)] The whisker $\nu$ is a path from the basepoint of $\Sigma\times I$ to the basepoint of $N\times I$, which also determines the inclusion $F_*:\pi_1(\Sigma\times I)\rightarrow \pi_1(N\times I)$;
        \item[(b)] $\gamma_1^p$ and $\gamma_2^p$ are arbitrary paths in $N\times I$ from the basepoint to $p$ along two different sheets;
        \item[(c)] $\epsilon(p)\in\{\pm1\}$ is $+1$ if the orientation of $M$ at p matches the one induced by the orientation of $N'$ on the two sheets (following the order); 
        \item[(d)] $\alpha(p)$ is the loop $\nu\gamma_1^p(\gamma_2^p)^{-1}\nu^{-1}\in\pi_1(N\times I)$ and $\overline{\alpha(p)}$ is the element of $G\backslash\pi_1(N\times I)/G$ given by the class; and
        \item[(e)] $\thicksim$ refers to the equivalence relations $\overline{1} = 0$ and $\overline{g}^{-1}=-\overline{g}$ in $\mathbb{Z}G\backslash\pi_1(N\times I)/G$.
    \end{enumerate}
    
\end{definition}
This definition generally follows the definition of Wall's invariant. And the equivalence relation comes from the facts that for any self-intersection point, if we change the order of two sheets, $\alpha(p)$ gives the inverse element and $\epsilon(p)$ becomes the opposite number, and that a cusp homotopy generates a single self-intersection point corresponding to the trivial element in $\pi_1(N\times I)$. We have the following proposition:
\begin{proposition}
    Let $\Sigma$, $N$, and $F$ be defined as above. The self-intersection number $\mu$ has the following properties.
    \begin{enumerate}
        \item The self-intersection number $\mu(F)$ is unchanged under homotopies in the interior of $N\times I$. 
        \item A different choice of whisker $\nu'$ for $F$ results in conjugation of $\mu(F)$ by the element $\nu'\nu^{-1}$ of $\pi_1(N\times I)$.
    \end{enumerate}
\end{proposition}
\begin{proof}
    The property (1) is analogous to the property of Wall's intersection number, and here we just need to check that our intersection number is well-defined. For difference choices of $\gamma_1^p$ and $\gamma_2^p$, they may result in difference loops $\alpha(p)=\nu\gamma_1^p(\gamma_2^p)^{-1}\nu^{-1}$. Let $\eta_1^p$, $\eta_2^p$ be another choice of paths in $N'$ from the basepoint to $p$, then $\gamma_1^p(\eta_1^p)^{-1}$, $\gamma_2^p(\eta_2^p)^{-1}$ are two loops in $N\times I$, thus $\eta_1^p=g_1\gamma_1^p$, $\eta_2^p=g_2\gamma_2^p$, where $g_1,g_2\in G$ and so $\nu\gamma_1^p(\gamma_2^p)^{-1}\nu^{-1}\in\pi_1(N\times I)$ and $\nu\eta_1^p(\eta_2^p)^{-1}\nu^{-1}\in\pi_1(N\times I)$ are in the same coset in $G\backslash\pi_1(N\times I)/G$, which means $\overline{\alpha(p)}$ is well-defined and $\mu(F)$ is unchanged under homotopy.\par
    The proof of (2) is straightforward, just notice that $$\nu'\gamma^p_1\gamma^p_2\nu'^{-1}=\nu'\nu^{-1}\nu\gamma^p_1\gamma^p_2\nu^{-1}(\nu'\nu^{-1})^{-1}.$$
\end{proof}

Similar to the Wall's intersection number, we have the following theorem.

\begin{theorem}\label{thm2.4}
    An immersion $F:\Sigma\times I\rightarrow N\times I$ is homotopic to an embedding relative to the boundary if and only $\mu(F)=0$
\end{theorem}

\begin{proof}[Proof of Theorem \ref{thm2.4}]
    If $F$ is isotopic to an embedding, then there is no self-intersection, hence $\mu(F)=0$. On the other hand, if $\mu(F)=0$, by the definition $\mu$, every intersection point $p$ of $F$ can either be paired by another intersection point $q$ such that $\epsilon(p)\overline{\alpha(p)}+\epsilon(q)\overline{\alpha(q)}=0$, or $\overline{\alpha(p)}=0$.\par
    For the first case, $\overline{\alpha(p)}=-\overline{\alpha(q)}$. Let $\gamma_i$, for $i\in\{1,2,3,4\}$ be disjointly embedded arcs in $\Sigma\times I$ such that $\gamma_1$ goes from the base point of $F$ to $p$ along the second sheet of $F$ at $p$, $\gamma_2$ goes from $p$ to itself leaving on the second sheet of $F$ at $p$ and returning on the first, $\gamma_3$ goes from $p$ to $q$, leaving o the first sheet of $i$ at $p$ and ending on the second sheet of $F$ at $q$, and $\gamma_4$ goes from $q$ to itself, leaving on the second sheet of $F$ at $q$ and returning on the first. And then $\epsilon(p)\overline{\alpha(p)}=\epsilon(p)\overline{\nu\gamma_1\gamma_2\gamma_1^{-1}\nu^{-1}}$, and similarly, $\epsilon(q)\overline{\alpha(q)}=\epsilon(q)\overline{\nu\gamma_1\gamma_2\gamma_3\gamma_4\gamma_3^{-1}\gamma_2^{-1}\gamma_1^{-1}\nu^{-1}}$.Then by appropriate choice of $\gamma_3$ and $\gamma_4$, we could obtain $\nu\gamma_1\gamma_2\gamma_1^{-1}\nu^{-1}=\nu\gamma_1\gamma_2\gamma_3\gamma_4\gamma_3^{-1}\gamma_2^{-1}\gamma_1^{-1}\nu^{-1}$ as elements in $\pi_1(N\times I)$. In other words, $(\nu\gamma_1\gamma_2^{-1}\gamma_1^{-1}\nu^{-1})(\nu\gamma_1\gamma_2\gamma_3\gamma_4\gamma_3^{-1}\gamma_2^{-1}\gamma_1^{-1}\nu^{-1})=\nu\gamma_1\gamma_3\gamma_4\gamma_3^{-1}\gamma_2^{-1}\gamma_1^{-1}\nu^{-1}$ is trivial in $\pi_1(N\times I)$. This is a basepoint-changing conjugation away from the loop $\gamma_3\gamma_4\gamma_3^{-1}\gamma_2^{-1}$, and so the loop is null-homotopic in $N\times I$. Thus, $p$ and $q$ can be paired by a Whitney disk.\par
    For the second case, notice that $\overline{\alpha(p)}=\overline0=\overline1$, and a cusp homotopy generates a self-intersection $q$ corresponding to the trivial element in $\pi_1(N\times I)$. And so $p$ and $q$ can be regarded as two intersection points in case (a), which can be paired by a Whitney disk.\par
    Since the manifold is 3-dimensional and codimension is also 3, by the Whitney move, all the intersections can be eliminated in pairs \cite{Milnor1965}, and so $F$ is homotopic to an immersion with no intersection points, which is an embedding.
\end{proof}

Since $\mu$ is a homotopy invariant, we may now define $\mu$ on the homotopy classes $\mathscr{H}^f$, which induces a map $\mu:\mathscr{H}^f\rightarrow\mathbb{A}_{[f]}$. Now, let us prove the main theorem of this section.

\begin{proof}[Proof of theorem \ref{thm2.3}]
    The necessity is straightforward. An isotopy has no self-intersection points and thus the intersection number vanishes.\par
    For sufficiency, if $\mu([H])=0$, by a small perturbation we may assume that $H$ is an immersion such that $\mu(H)=0$, then by Theorem \ref{thm2.4}, $H$ is homotopic to an embedding $G:\Sigma\times I\rightarrow N\times I$ relative to the boundary. In the general case, the embedding may not be an isotopy, for it may not preserve the "$I$" coordinate, which is called a "concordance'. However, we have the following theorem.
    \begin{theorem}[\cite{Hudson1970}]
        If $M$ and $Q$ are compact smooth manifolds, let $F:M\times I\rightarrow Q\times I$ be an allowable smooth concordance. If $\textup{dim}M\leq\textup{Q}-3$, then there is an ambient diffeotopy $J$ of $Q\times I$, fixed on $Q\times 0$ such that $J_1F=F_0\times id$.
    \end{theorem}
    Since $G:\Sigma\times I \rightarrow N\times I$ has codimension 3, we can apply this theorem to learn that $G$ is homotopic to an isotopy, which means $H$ is homotopic to an isotopy relative to the boundary.
\end{proof}

The following Figure \ref{fig:figure1} shows how an element of the quotient fundamental group acts on the homotopic surface. For $g\in\pi_1N$ and $i:\Sigma\hookrightarrow N$, we define $g\cdot i:\Sigma\hookrightarrow N$ as follows. Firstly, note that the normal bundle of $i$ is $3-$dimensional, so its meridian $m$ is a $2-$sphere; we choose it over a point near the basepoint $p\coloneq i(e)$ and orient it according to the orientation of $\Sigma$ and $N$. We then define $g\cdot i$ as an ambient connected sum of $i$ and $m$ along a tube following an arc representing $g$, where the arc starts and ends near $p$ and has interior disjoint from $i$.

\begin{figure}
    \begin{overpic}[scale=0.7]{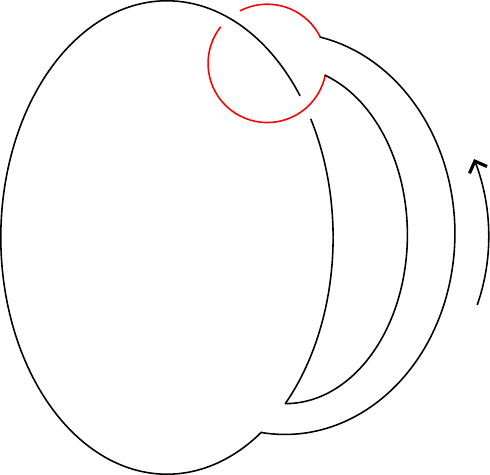}
		\put(40,67){\textcolor[rgb]{1,0,0}{$m$}}
		\put(30,15){\large{$g\cdot i$}}
        \put(105,50){$g$}
	\end{overpic}
    \caption{The action of $g$ on $i$}
    \label{fig:figure1}
\end{figure}

Similarly, linear combinations $\Sigma_in_ig_i$ act by multiple connected sums along $g_i$ into copies of $m$ for $n_i$ times, and if $n<0$ just reverse the orientation. The relation $g=g^{-1}$ and $0=1$ is easy to check, see Figure \ref{fig:figure2}.

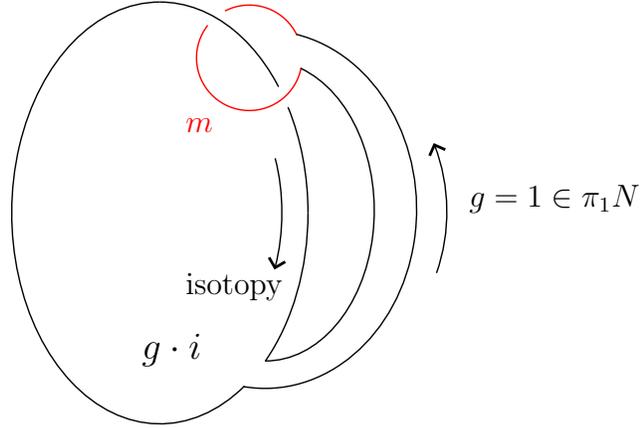
\begin{figure}
    \begin{overpic}[scale=0.7]{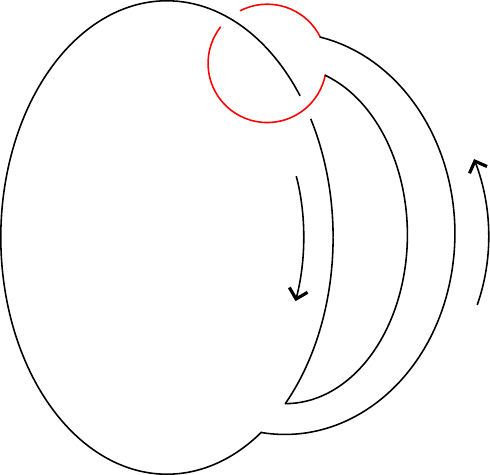}
		\put(40,67){\textcolor[rgb]{1,0,0}{$m$}}
		\put(30,15){\large{$g\cdot i$}}
        \put(105,50){$g=1\in\pi_1N$}
        \put(40,30){isotopy}
	\end{overpic}
    \caption{The relation $1=0$ realized by isotopies}
    \label{fig:figure2}
\end{figure}

If the group $G=i_*(\pi_1\Sigma)$ is not trivial, the action of an element $g\in G\subset\pi_1N$ would also be trivial. Geometrically, this can be seen by reversing the action along a representation of $g$ in $i$, see Figure \ref{fig:figure3}.

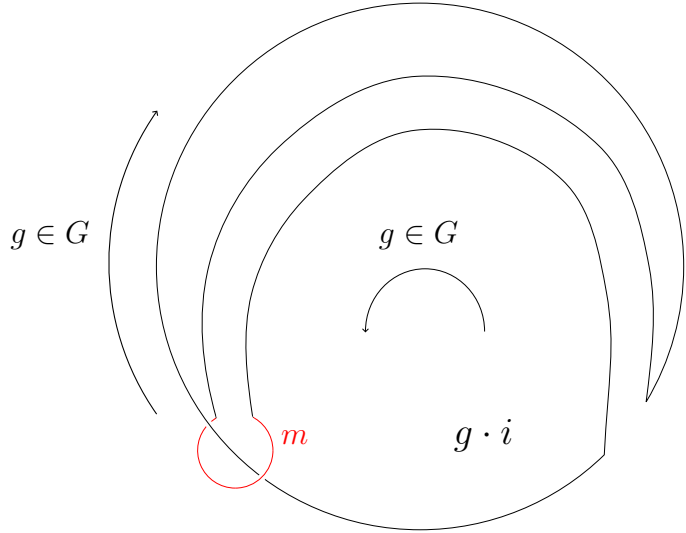
\begin{figure}
    \begin{overpic}[scale=0.5]{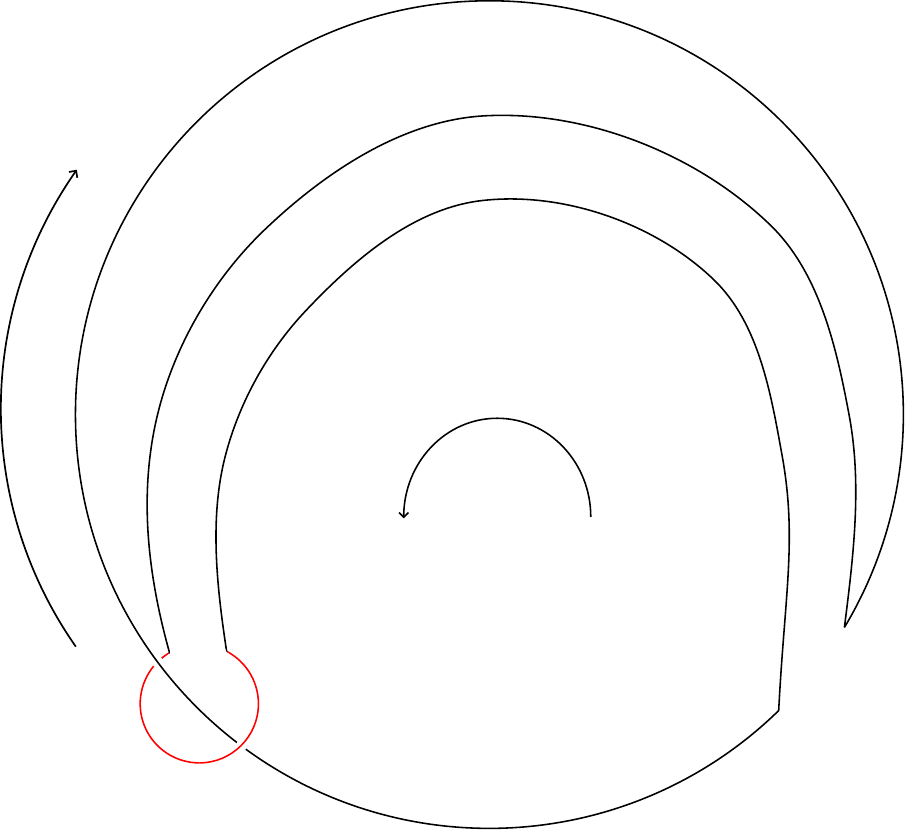}
		\put(30,15){\textcolor[rgb]{1,0,0}{$m$}}
		\put(60,15){\large{$g\cdot i$}} 
        \put(-17,50){$g\in G$}
        \put(47,50){$g\in G$}
	\end{overpic}
    \caption{The relation $0=g\in G$ realized by isotopies}
    \label{fig:figure3}
\end{figure}

\section{Affine Action of Self-homotopies}

From the last section, we learn that if there are two homotopic embeddings $f$ and $f'$, and a homotopy $H$ from $f$ to $f'$ such that $\mu([H])=0$, then $H$ is homotopic to an isotopy relative to the boundary, which means that $f$ is isotopic to $f'$. However, the homotopy between them is not unique. If $G$ is a self-homotopy of $f$, in other words, $G$ is a homotopy from $f$ to $f$ itself, then the concatenation of $G$ and $H$ gives another homotopy from $f$ to $f'$. So, even if $\mu([H])\neq 0$, we cannot claim that $f$ is not isotopic to $f'$. This inspires us to study the influence of self-homotopies on our invariant.\par
Let's start with some definitions.
\begin{definition}\label{def3.1}
    Let $\mathscr{H}_0^f$ be the subspace of $\mathscr{H}^f$ such that for every $[H]\in\mathscr{H}_0^f$, $H_0=H_1=f$. We call $\mathscr{H}_0^f$ the set of self-homotopies of $f$.
\end{definition}

A basic fact is that under concatenation, $\mathscr{H}^f_0$ forms a group.

\begin{proposition}\label{prop3.2}
    For $[H],[H']\in\mathscr{H}_0^f$, define $[H]\cdot[H']=[H\cdot H']$, where $$H\cdot H'(x,t)=\left\{
    \begin{aligned}
    & H(x,2t),0\leq t\leq 1/2,\\
    & H'(x,2t-1),1/2<t\leq 1
    \end{aligned}
    \right.
    $$
    Then $\mathscr{H}_0^f$ forms a group under the product.
\end{proposition}

\begin{proof}
    Since $(H\cdot H')_0=H_0=f$ and $(H\cdot H')_1=H'_1=f$, the product is closed. Let $E(x,t)=f(x), \forall t\in I$, then for every homotopy class $[H]\in\mathscr{H}^f_0$ consider the homotopy $J:\Sigma\times I\times I\rightarrow \Sigma\times I$, $$J(x,t,s)=\left\{
    \begin{aligned}
    & E(x,2t),0\leq t\leq \frac{1-s}{2},\\
    & H(x,\frac{2}{1+s}t-\frac{1-s}{1+s}),\frac{1-s}{2}<t\leq 1
    \end{aligned}
    \right.
    $$
    $J$ gives a homotopy between $E\cdot H$ and $H$. Thus $[E]\cdot [H]=[H]$ and similarly $[H]\cdot[E]=[H]$, so $[E]$ is the unit element. Moreover, let $H$ be a self-homotopy of $f$, define $-H$ as $-H(x,t)=H(x,1-t)$, then it's also similar to see $H\cdot(-H)\sim(-H)\cdot H\sim E$. So $\mathscr{H}^f_0$ forms a group.
\end{proof}

To give the action of $\mathscr{H}^f_0$ on $\mathscr{H}^f$ and $\mathbb{A}_{[f]}$, we need to consider the CW decomposition of $\Sigma$, in other words, let $\Sigma=e^0\cup (\cup_\alpha e^1_\alpha)\cup e^2$. Denote by $\Sigma^1$ the 1-skeleton $e^0\cup (\cup_\alpha e^1_\alpha)$. Consider the self-homotopies on the skeletons of $\Sigma$.

\begin{definition}
    Suppose $A$ is a closed submanifold of $\Sigma$. Let $\mathscr{H}^f_0(A,N)$ denote the space of homotopy classes of self-homotopies $H:A\times I\rightarrow N\times I$ relative to homotopy fixing the boundary, such that $H_0=H_1=f|_{A}$.
\end{definition}

By definition, we have $\mathscr{H}^f_0=\mathscr{H}^f_0(\Sigma,N)$.\par

Consider the map
$$\mathscr{H}^f_0(\Sigma,N)\xrightarrow{r_2}\mathscr{H}_0^f(\Sigma^1,N)\xrightarrow{r_1}\mathscr{H}^f_0(e_0,N)\xrightarrow{r_0}*,$$

where the maps $r_1$, $r_2$ are given by restrictions. Just as Proposition \ref{prop3.2}, it is straightforward to prove that $\mathscr{H}^f_0(\Sigma^1,N)$ and $\mathscr{H}^f_0(e^0,N)$ are groups under concatenation, and restrictions $r_2,r_1$ are group homomorphisms. For simplicity, we will denote by $H$ a self-homotopy of $\Sigma$, $H^1$ a self-homotopy of $\Sigma^1$, and $H^*$ a self-homotopy of the basepoint.

\begin{definition}
    Let $\mathscr{G}_2$ be $\ker(r_2)$, $\mathscr{G}_1$ be $\ker(r_1)\cup\textup{im}(r_2)$ and $\mathscr{G}_0$ be $\textup{im}(r_1\circ r_2)$. 
\end{definition}

In other words, $\mathscr{G}_2$ consists of all the homotopies that fix the 1-skeleton. $\mathscr{G}_1$ consists of all the homotopies of 1-skeletons that fix the basepoint and can be extended to self-homotopies of $\Sigma$. And $\mathscr{G}_0$ consists of all the homotopies of the basepoint that can be extended to self-homotopies of $\Sigma$. Since $r_2,r_1$ are group homomorphisms, we have

\begin{proposition}\label{prop3.7}
    $\mathscr{G}_2$ is a normal subgroup of $\mathscr{H}^f_0(\Sigma,N)$, and $\mathscr{G}_1$ is a normal subgroup of $\textup{im}(r_2)$.
\end{proposition}

Hence, there is a group isomorphism $(\mathscr{H}^f_0/\mathscr{G}_2)/\mathscr{G}_1\cong\mathscr{G}_0$. To give the commutative diagram in Theorem $\ref{thm1.1}$, we consider the following decomposed commutative diagram:

\begin{center}
\begin{tikzcd}
& \mathscr{H}^f \arrow[r,"q_2"] \arrow[d,"\mu"]
& \mathscr{H}^f/\mathscr{G}_2 \arrow[r,"q_1"] \arrow[d,"\mu^1"]
& (\mathscr{H}^f/\mathscr{G}_2)/\mathscr{G}_1 \arrow[r,"q_0"] \arrow[d,"\mu^*"]
& ((\mathscr{H}^f/\mathscr{G}^f_2)/\mathscr{G}_1)/\mathscr{G}_0 \arrow[d,"\overline\mu"]\\
& \mathbb{A}_{[f]} \arrow[r]
& \mathbb{A}_{[f]}/\mathscr{G}_2\eqqcolon\mathbb{A}_f^1 \arrow[r]
& \mathbb{A}_f^1/\mathscr{G}_1\eqqcolon\mathbb{A}_f^* \arrow[r]
& \mathbb{A}_f^*/\mathscr{G}_0\eqqcolon\mathbb{A}_{f}
\end{tikzcd}
\end{center}

where $\mathscr{G}_2,\mathscr{G}_1$ and $\mathscr{G}_0$ act on the corresponding sets respectively, $q_2,q_1$ and $q_0$ are projective maps, and $\mu^1,\mu^*$, and $\overline\mu$ are induced map of $\mu$ on the orbit spaces. We will show that $((\mathscr{H}^f/\mathscr{G}^f_2)/\mathscr{G}_1)/\mathscr{G}_0\cong\mathscr{H}^f/\mathscr{H}^f_0$. In section 4, we will further prove that the induced map $\overline\mu$ is a bijection.

In the following discussion, we will discuss the problem of extending a homotopy several times. Thus, for convenience, we state the homotopy extension theorem here.
\begin{lemma}[Homotopy Extension Theorem, see \cite{Hatcher2002}]\label{lemma3.8}
    Any CW pair $(X,A)$ has the homotopy extension property, namely, for any map $F:X\rightarrow Y$ and a homotopy $h:A\times I\rightarrow Y$ of $A$ satisfying $h_0=F|_{A}$, there exists a homotopy $H:X\times I\rightarrow Y$ of $X$ such that $H|_{A\times I}=h$.
\end{lemma}

\subsection{Action of Self-homotopies on $\mathscr{H}^f$}

Firstly, let's consider the action of $\mathscr{H}^f_0$ on the homotopies.
Notice that we may extend the product to any homotopy $H$ and $G$ as long as $H_1=G_0$, in other words, if $H$ and $G$ are two homotopies, such that $H_1=G_0$, then define $$H\cdot G(x,t)=\left\{
    \begin{aligned}
    & H(x,2t),0\leq t\leq 1/2,\\
    & G(x,2t-1),1/2<t\leq 1
    \end{aligned}
    \right.
    $$
Then, if we let $[H]\in\mathscr{H}^f_0$ and $[G]\in\mathscr{H}^f$, we obtain a group action.

\begin{proposition}\label{prop3.6}
     The left group action $\Phi:\mathscr{H}_0^f\times \mathscr{H}^f\rightarrow \mathscr{H}^f$, defined by $\Phi([H],[G])=[H\cdot G]$, is well-defined and faithful.
\end{proposition}

\begin{proof}
    If $G$ is a homotopy of $f$ such that $G_1=f'$ is another embedding, then $H\cdot G$ is a homotopy from $f$ to $f'$, and so the group action is closed. And for $\forall H,H'\in\mathscr{H}^f_0$ and $G\in\mathscr{H}^f$, it is easy to show $(H\cdot H')\cdot G=H\cdot(H'\cdot G)$, so the group action is well-defined.\par
    Moreover, if $H\in\mathscr{H}^f_0$ and $G\in\mathscr{H}^f$ satisfies $[H\cdot G]=[G]$, then $H\cdot G\sim G$, hence $H\cdot G\cdot(-G)\sim G\cdot(-G)$, thus $H\sim E$, which means $[H]=1$, so the group action is faithful.
\end{proof}

Denote by $\mathscr{H}^f/\mathscr{H}^f_0$ the set of orbits under the action of $\mathscr{H}^f_0$. Recall that we define the projection $p:\pi_0\textup{Emb}(\Sigma,N)\twoheadrightarrow\pi_0\textup{Map}(\Sigma,N)$. We now obtain the following important bijection.

\begin{theorem}\label{thm3.8}
    Let $P$ denote the map
    \begin{align*}
    P:p^{-1}([f]) & \rightarrow\mathscr{H}^f/\mathscr{H}^f_0\\
    g & \mapsto \overline{[H]}
    \end{align*}
    where $H$ is a homotopy connecting $f$ and $g$ and $\overline{[H]}$ denote the orbit class of the homotopy class of $H$. Then $P$ is a bijection.
\end{theorem}

\begin{proof}
    To show that $P$ gives a well-defined map from $p^{-1}([f])$ to $\mathscr{H}^f/\mathscr{H}^f_0$, for every $g\in p^{-1}([f])$ and two homotopies $G$ and $G'$ connecting $f$ and $g$, notice that $G\cdot(-G')$ is a homotopy from $f$ to $f$ itself, so $[G\cdot(-G')]\in\mathscr{H}^f_0$, and $(G\cdot(-G'))\cdot G'\sim G$, so $G$ and $G'$ lie in the same orbit, which means $P$ is well-defined. For surjectivity, just notice that every homotopy of $f$ is a homotopy that connects $f$ and another embedding $f'$. To show injectivity, let two embeddings $g$ and $g'$ satisfy $P(g)=P(g')$. This means that there exist two homotopies $G$ and $G'$ such that $G$ connects $f$ and $g$, while $G'$ connects $f$ and $g'$, and there exists a self homotopy $H$ of $f$ such that $H\cdot G\sim G'$. So $-G'\cdot H\cdot G\sim E$ which is an isotopy, and since $-G'\cdot H\cdot G$ is a homotopy connecting $g'$ and $g$, we learn that $g$ is isotopic to $g'$.
\end{proof}

Now we consider the action of $\mathscr{G}_2,\mathscr{G}_1$ and $\mathscr{G}_0$.

\begin{definition}
    Define the action $\mathscr{G}_2\times \mathscr{H}^f\rightarrow\mathscr{H}^f$ as the restriction of $\Phi$ on $\mathscr{G}_2$ which is a normal subset of $\mathscr{H}_0^f$. Denote by $\mathscr{H}^f/\mathscr{G}_2$ the orbit space and $q_2$ the quotient map $\mathscr{H}^f\rightarrow\mathscr{H}^f/\mathscr{G}_2$.
\end{definition}

\begin{definition}
    Let $[H^1]\in\textup{im}(r_2)$ be a self-homotopy of $\Sigma^1$, and let $[H]$ be an extension of $[H^1]$ to a self-homotopy of $\Sigma$. Denote by $\Psi$ the action of $\textup{im}(r_2)$ on $\mathscr{H}^f/\mathscr{G}_2$, such that $\Psi([H^1],[G])=q_2([H\cdot G])$.
\end{definition}

\begin{proposition}
    $\Psi$ is a well-defined faithful group action, and there exists a bijection $(\mathscr{H}^f/\mathscr{G}_2)/\textup{im}(r_2)\cong\mathscr{H}^f/\mathscr{H}^f_0$
\end{proposition}

\begin{proof}
    For $[H^1]\in\textup{im}(r_2)$, if there are two homotopies $[H],[H']\in\mathscr{H}^f_0$ such that $r_2([H])=r_2([H'])$, then $r_2([H]\cdot[H']^{-1})=1$, thus $[H]\cdot[H']^{-1}\in\mathscr{G}_2$. Thus $q_2([H']\cdot[G])=q_2(([H]\cdot[H']^{-1})\cdot[H']\cdot[G])=q_2([H]\cdot[G])$. So $\Psi$ is well-defined. Moreover, for $[H^1_1],[H^1_2]\in\textup{im}(r_2)$, $\Psi([H^1_1],\Psi([H^1_2],[G]))=[H^1_1]\cdot[H^1_2]\cdot[G]=\Psi([H^1_1]\cdot[H^1_2],[G])$, therefore, $\Psi$ is a group action. Finally, if $\Psi([H^1],[G])=q_2([H]\cdot[G])=q_2[G]$ for $[G]\in\mathscr{H}^f$, then there exists $[K]\in\mathscr{G}_2$ such that $[K]\cdot[H]\cdot[G]=[G]$, thus $[H]=[K]^{-1}\in\mathscr{G}_2$, which means that $[H^1]=r_2([H])=1$, therefore $\Psi$ is a faithful action.\par
    Moreover, consider the map
    \begin{align*}
    Q:\mathscr{H}^f/\mathscr{G}_2 & \rightarrow\mathscr{H}^f/\mathscr{H}^f_0\\
    q_2[G] & \mapsto \overline{[G]},
    \end{align*}
    since $\mathscr{G}_2$ is a normal subgroup of $\mathscr{H}^f_0$, $Q$ is a well-defined surjection. For every $\overline{[G]}\in\mathscr{H}^f/\mathscr{H}^f_0$, $Q(q_2[G]')=\overline{[G]}$ if and only if there exists $[H]\in\mathscr{H}^f_0$ such that $[H]\cdot[G']=[G]$, and so $Q^{-1}(\overline{[G]})=q_2(\Phi(\mathscr{H}^f_0,[G]))=\Psi(\textup{im}(r_2),[G])$. So $Q$ induces a bijection $\tilde{Q}:(\mathscr{H}^f/\mathscr{G}_2)/\textup{im}(r_2)\rightarrow\mathscr{H}^f/\mathscr{H}^f_0$.
\end{proof}

Similarly, since $\mathscr{G}_1$ is a normal subgroup of $\textup{im}(r_2)$, we may define an action $\mathscr{G}_1\times(\mathscr{H}^f/\mathscr{G}_2)\rightarrow(\mathscr{H}^f/\mathscr{G}_2)$ as the restriction of $\Psi$. Let $[H^*]\in\textup{im}(r_1\circ r_2)$ be a self-homotopy of $e^0$, and let $[H^1]$ be an extension of $[H^*]$ to a self-homotopy of $\Sigma^1$. We can also define an action $\Gamma$ of $\textup{im}(r_1\circ r_2)=\mathscr{G}_0$ on $(\mathscr{H}^f/\mathscr{G}_2)/\mathscr{G}_1$, such that $\Gamma([H^*],[G])=q_1(\Psi([H^1],[G])$. And by a similar discussion, it can be shown that there exists the bijection $R:((\mathscr{H}^f/\mathscr{G}_2)/\mathscr{G}_1)/\mathscr{G}_0\rightarrow\mathscr{H}^f/\mathscr{H}^f_0$.

\subsection{Action of Self-homotopies on $\mathbb{A}_{[f]}$}

Now we consider the actions on the target set. Let us start with a more detailed discussion of $\mathscr{G}_2,\mathscr{G}_1$ and $\mathscr{G}_0$.

\begin{theorem}
    There exists a group isomorphism $\mathscr{G}_2\cong\pi_3(N)$.
\end{theorem}

\begin{proof}
    Let $[G]\in\mathscr{G}_2$. So $G$ is a self-homotopy of $f$ fixing the 1-skeleton $\Sigma^1$. Consider the unit homotopy $E=f\times\textup{id}$. By definition of $\mathscr{G}_1$, $G$ agrees with $E$ on the 2-skeleton $\Sigma^1\times I$ of the 3-manifold $\Sigma\times I$. Let $B^G$ be the immersion of the closed 3-ball $G|_{D^2\times I}:D^2\times I\rightarrow N\times I$, and $B^E$ be $E|_{D^2\times I}=f|_{D^2}\times\textup{id}$, where $D^2=\overline{e^2}$ is the closure of the 2-cell. Gluing $B^G$ and $-B^E$ along their boundaries produces a map of 3-spheres $S^G\coloneqq B^G\cup (-B^E):S^3\rightarrow N\times I$. And we define the map $\mathscr{G}_1\rightarrow\pi_3(N)=\pi_3(N\times I)$ as $[G]\mapsto[S^G]$. By the definition of product of $\mathscr{G}_2$ and $\pi_3(N)$, $[G\cdot G']\mapsto[S^G]\cdot[S^{G'}]$. If $[S^G]$ is trivial, then $B^G$ is homotopic to $B^E$, which means that $[G]$ is trivial. Thus, the map is an injection. On the other hand, for every $[S]\in\pi_3(N)$, by a homotopy, $S$ can be written as a union of maps $B\cup (-B^E):S^3=D^2\times I\cup D^2\times I\rightarrow N\times I$, then we may construct a homotopy $G\in\mathscr{G}_2$ such that $B^G=B$, which means that $S^G=S$.
\end{proof}

\begin{definition}
    Define $\phi:\mathscr{G}_2\rightarrow\mathbb{A}_{[f]}$ by $[G]\rightarrow\mu([S^G])+\overline{\lambda}_N(S^G,f(\Sigma))$, where $\overline\lambda_N$ is the class of the common intersection invariant $\lambda_N$ in the quotient group $\mathbb{A}_{[f]}$. Denote by $\mathbb{A}_{f_1}$ the quotient group $\mathbb{A}_{[f]}/\textup{im}(\phi)$.
\end{definition}

\begin{lemma}
    If $[G]\in\mathscr{G}_2$, then $\mu([G])=\phi(S^G)$.
\end{lemma}

\begin{proof}
    By definition, $S^G$ is the union of 2 maps of 3-balls $B^G$ and $B^E$. Then, by the dimension condition, all self-intersections contributing to $\mu([G])$ come from the 3-cell $B^G$. And the self-intersections contributing to $\mu(S^G)$ come from $B^G$ and the intersections between $B^G$ and $-B^E$. Therefore,
    $$\mu(S^G)=\mu([G])+\overline{\lambda}(B^G,-B^E)=\mu([G])+\overline{\lambda}([G],f\times \textup{id})$$
    And we have $\overline{\lambda}_N(S^G,f)=\overline{\lambda}(S^G,f\times\textup{id})=\overline{\lambda}([G],f\times\textup{id})$, where the first 2 equations hold due to the dimension condition, and the last equation holds because $f(D^2)\times I$
    is contractible, therefore it admits a trivial normal bundle in N,  so we can slightly move it along some nonzero normal bundle to remove the self-intersections, thus we obtain
    $$\mu([G])=\mu(S^G)+\overline{\lambda}_N(S,f)=\phi(S).$$
\end{proof}

\begin{lemma}\label{lemma3.12}
    The map $\mu:\mathscr{H}^f\rightarrow\mathbb{A}_{[f]}$ induces a well-defined map $\mu^1:\mathscr{H}^f/\mathscr{G}_2\rightarrow \mathbb{A}_{f_1}$. And if $q_2[H]\in\mathscr{H}^f/\mathscr{G}_2$ satisfies $\mu^1(q_2[H])=0$, then $q_2[H]=q_2[E]$.
\end{lemma}

\begin{proof}
    For $q_2[H]=q_2[H']\in\mathscr{H}^f/\mathscr{G}_2$, by definition there exists $[G]\in\mathscr{G}_2$ such that $[G]\cdot [H]=[H']\in\mathscr{H}^f$. And by definition of the intersection number, $\mu([G\cdot H])=\mu([G])+\mu([H])$ since $[G]\in\mathscr{G}_2$ fixes the basepoint, hence, $\mu([G])\in\phi(\mathscr{G}_2)$, and so $\mu([H])$ and $\mu([H'])$ are in the same orbit of $\mathbb{A}_{f_1}=\mathbb{A}_{[f]}/\phi(\mathscr{G}_2).$ Thus the induced map $\mu^1$ is well defined. Moreover, if $\mu^1(q_2[H])=0$, then there exists $[G]\in\mathscr{G}_2$ such that $\mu([H])=\mu([G])$. Thus, $0=\mu([H])-\mu([G])=\mu([H]\cdot[-G])=\mu([H\cdot(-G)])$, so by Theorem \ref{thm2.3}, $(H\cdot(-G))$ is homotopic to $E$, thus $[H]\cdot(-[G])=[E]$, which means $q_2[H]=q_2[E]$.
\end{proof}

We have a similar discussion for $\mathscr{G}_1$.

\begin{definition}
    For a surface $\Sigma$ with genus $g$, the 1-skeleton $\Sigma^1\cong\bigvee_{2g}S^1$. Thus, a homotopy of $\Sigma^1$ fixing the basepoint corresponds to a based immersion $\bigvee_{2g}S^2\rightarrow N\times I$, in other words $\ker(r_1)\cong\pi_2(N)^{2g}$. For an element $a\in\pi_2(N)^{2g}$, denote by $[H^1_a]\in\ker(r_1)$ the homotopy of $\Sigma^1$ corresponding to $a$. Then by the homotopy extension theorem (Lemma \ref{lemma3.8}), $H^1_a:\Sigma^1\times I\rightarrow N\times I$ can be extended to a homotopy $H_a:\Sigma\times I\rightarrow N\times I$ such that $H_a(x,0)=f(x),\forall x\in\Sigma$ and $H_a|_{\Sigma^1\times I}(x,t)=H_a^1(x,t)$. Thus, $H_a|_{\Sigma\times\{1\}}$ agrees with $f$ on the boundary of the 2-cell $e^2$. Gluing $H_a|_{D^2\times\{1\}}$ and $f|_{D^2}$ along the boundary induces a map $S_a:S^2\rightarrow N\times I$. Hence, the procedure implies a map
    \begin{align*}
        \iota:\pi_2(N)^{2g} & \rightarrow \pi_2(N)\\
        a & \mapsto [S_a].
    \end{align*}
\end{definition}

\begin{theorem}
    The map $\iota$ is a well-defined group homomorphism and $\mathscr{G}_1\cong\ker(\iota)$.
\end{theorem}

\begin{proof}
    Obviously, $\iota(1)=1$, and by the construction of $S_a$, since the two cell $e^2$ are contractible, for $a,b\in\pi_2(N)^{2g}$, $[S_{a\cdot b}]=[S_a]\cdot[S_b]$, thus $\iota(ab)=\iota(a)\iota(b)$. Let $H_a^1$ be a homotopy of $\Sigma^1$, and let $H_a,H_a'$ be two extensions of $H_a^1$. Define $S_a=H_a(D^2,1)\cup_{\partial} f(D^2)$ and $S_a'=H_a'(D^2,1)\cup_\partial f(e^2)$, then we just need to prove $[S_a]=[S_a']$ in $\pi_2N$. Notice that $[S_a]\cdot[S_a']^{-1}=[S_a\cup(-S_a')]=[H_a(D^2,1)\cup_\partial H_a'(D^2,1)]$. And since $H_a(D^2,1),H_a'(D,1)$ are all homotopic to $f(D^2)$ relative to the boundary, $H_a(e^2,1)\cup_\partial H_a'(e^2,1)$ is null-homotopic in $N$, therefore, $[S_a]=[S_a']$ and so the map $\iota$ is a well-defined group homomorphism.\par
    Moreover, since $\mathscr{G}_1=\textup{im}(r_2)\cup\ker(r_1)$, $\mathscr{G}_1$ is a subgroup of $\ker(r_1)$. By definition, for $[H^1]=r_2[H]\in\mathscr{G}_1$, $\iota([H^1])=[S]$ which is generated by gluing $H|_{D^2\times\{1\}}$ and $f|_{D^2}$. However, $[H]\in\mathscr{H}^f_0$, which means that $H|_{\Sigma\times\{1\}}=f$, so $\iota([H^1])=0$. Thus, $\mathscr{G}_1\cong\ker({\iota})$.
\end{proof}

\begin{definition}
    Define 
    \begin{align*}
    \psi:\mathscr{G}_1 & \rightarrow\mathbb{A}_{f_1}\\
    a & \mapsto\mu^1([H_a]).
    \end{align*}
    Denote by $\mathbb{A}_{f_*}$ the quotient group $\mathbb{A}_{f_1}/\psi(\mathscr{G}_1)$.
\end{definition}

In other words, $\psi=\mu^1|_{\mathscr{G}_1}$. And we have the following obvious fact.

\begin{proposition}
    $\psi$ is a group homomorphism.
\end{proposition}

Similarly, combining the map with the group action, we have:

\begin{lemma}\label{lemma3.17}
    The map $\mu^1:\mathscr{H}^f/\mathscr{G}_2\rightarrow\mathbb{A}_{f_1}$ induces a map $\mu^*:(\mathscr{H}^f/\mathscr{G}_2)/\mathscr{G}_1\rightarrow\mathbb{A}_{f_*}$. And if $q_1\circ q_2[H]\in(\mathscr{H}^f/\mathscr{G}_2)/\mathscr{G}_1$ satisfies $\mu^*(q_1\circ q_2[H])=0$, then we have $q_1\circ q_2[H]=q_2\circ q_1[E]\in(\mathscr{H}^f/\mathscr{G}_2)/\mathscr{G}_1$.
\end{lemma}

\begin{proof}
    For $q_1\circ q_2[H]=q_1\circ q_2[H']\in(\mathscr{H}^f/\mathscr{G}_2)/\mathscr{G}_1$, by definition there exists $q_2[G]\in\mathscr{G}_1$ such that $q_2([G]\cdot[H])=q_2[H']\in\mathscr{H}^f/\mathscr{G}_2$. Since $[G]\in\mathscr{G}_1$ fixes the basepoint, $\mu^1(q_2([G]\cdot[H]))=\mu^1(q_2[G])+\mu^1(q_2[H])$. And if $\mu^*(q_1\circ q_2[H])=0$, then there exists $q_1[G]\in\mathscr{G}_1$ such that $\mu^1(q_2[H])=\mu^1(q_2[G])$. So, $\mu^1(q_2[-G\cdot H])=0$, and by Lemma \ref{lemma3.12} and Theorem \ref{thm2.3} we learn that $q_2[G\cdot H]=q_2[E]$, thus, $q_1\circ q_2[H]=q_1\circ q_2[E]$.
\end{proof}

Finally, let's discuss the self-homotopies of the basepoint $\mathscr{G}_0$. A self-homotopy of the basepoint is just a loop from the basepoint to itself, which naturally corresponds to an element in $\pi_1N$. In other words, $\mathscr{H}^f_0(e^0,N)\cong\pi_1N$ \par

\begin{figure}
    \begin{overpic}[scale=0.5]{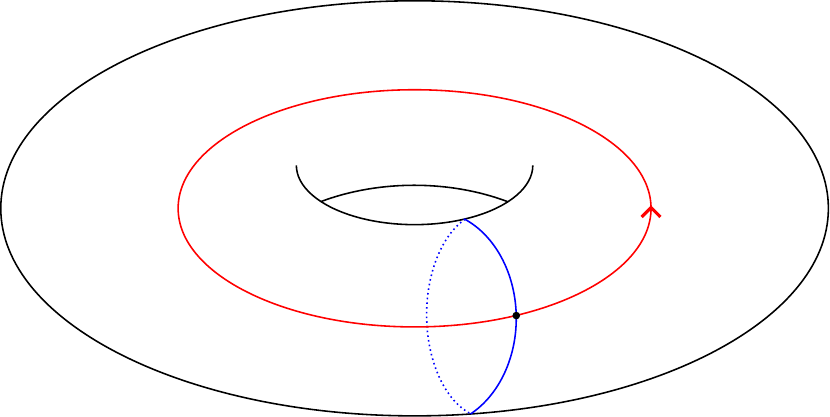}
		\put(80,25){\textcolor[rgb]{1,0,0}{$s$}}
		\put(63,15){\textcolor[rgb]{0,0,1}{$h$}} 
	\end{overpic}
    \caption{An example that the homotopy can be extended}
    \label{fig:figure4}
\end{figure}

\begin{figure}
    \begin{overpic}[scale=0.5]{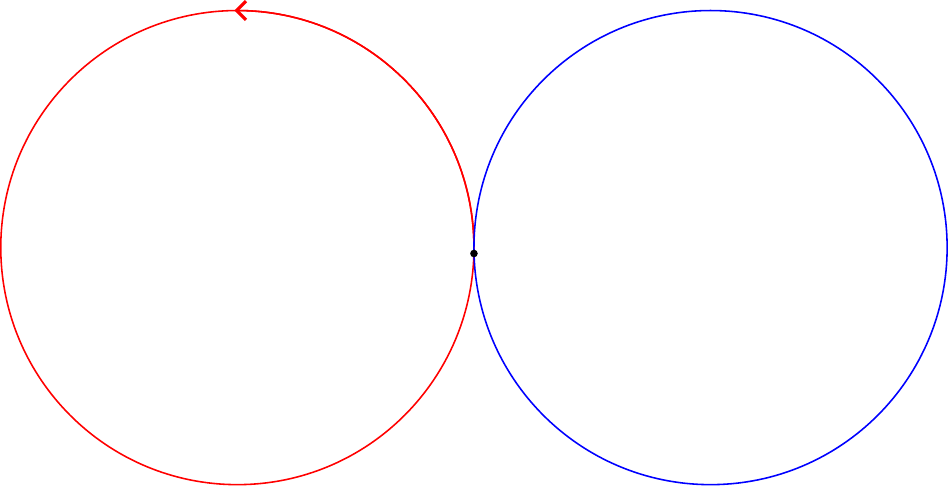}
		\put(20,45){\textcolor[rgb]{1,0,0}{$s$}}
		\put(70,45){\textcolor[rgb]{0,0,1}{$h$}} 
	\end{overpic}
    \caption{An example that the homotopy can not be extended}
    \label{fig:figure5}
\end{figure}

\begin{lemma}\label{lemma3.18}
    Let $g:S^1\rightarrow N$ be a continuous map. Fix a basepoint $e^0\in S^1$, then consider the restriction $r:\mathscr{H}^g_0(S^1,N)\rightarrow\mathscr{H}^g_0(e^0,N)$. Denote by $[H^*_s]$ the self-homotopy of the basepoint along $s\in\pi_1N$, $H^*_s\in\textup{im}(r)$ if and only if $s$ and $g$ commute as elements in $\pi_1N$, see Figure \ref{fig:figure4} and \ref{fig:figure5}.
\end{lemma}

\begin{proof}
    On the one hand, if $H^*_s\in\textup{im}(r)$, there exists a homotopy $[H]\in\mathscr{H}^g_0(S^1,N)$ such that $r[H]=[H^*_s]$, thus consider the homotopy $$G(x,t)=\left\{
    \begin{aligned}
    & s(3xt),0\leq x\leq 1/3,\\
    & H(3x-1,t),1/3<x\leq 2/3,\\
    & s(3t-3xt),2/3<x\leq 1
    \end{aligned}
    \right.
    $$
    and $[G(x,0)]=g$, $[G(x,1)]=sgs^{-1}$, therefore $sgs^{-1}=g$, which means $s$ and $g$ commutes.\par
    On the other hand, if $sgs^{-1}=g$, $sgs^{-1}g^{-1}=1$, thus, there exists a disk $D$ such that $[\partial D]=sgs^{-1}g^{-1}$. Define $H:S^1\times I\rightarrow N$, $H|_{S^1\times\{0\}}=H|_{S^1\times\{1\}}=g$ and $H|_{e^0\times I}=g$, and $H=D$ on the other area. Then $[H]\in\mathscr{H}^g_0(S^1,N)$ and $r[H]=[H^*_s]$ by construction.
\end{proof}

Let $C$ be the centralizer of $G=i_*(\pi_1\Sigma)\in\pi_1N$, i.e. $C$ consists of the elements commuting with all the elements in $G$. By Lemma \ref{lemma3.18}, $\textup{im}(r_1)\cong C$.\par

\begin{definition}
    For $s\in C$, denote by $H_s^1$ the homotopy of the 1-skeleton extending $s$. By the homotopy extension theorem, we may extend $H_s^1$ to a homotopy $H_s$ of $\Sigma$ up to $\mathscr{G}_2$. Thus, $H_s$ agrees with $f$ on the boundary of the 2-cell $D^2$. Define $\eta:C\rightarrow\pi_2N/\textup{im}(\iota)$ by $\eta(s)=[H_s|_{D^2\times{1}}\cup_\partial f|_{D^2}]$.
\end{definition}

\begin{theorem}
    $\eta$ is a group homomorphism and $\mathscr{G}_0\cong\ker(\eta)$.
\end{theorem}

\begin{proof}
    By the discussion above, a self-homotopy $s\in\pi_1N$ of the basepoint can be extended to a self-homotopy $H_s$ of $\Sigma$ if and only if $\eta(s)=0$, and the extension is unique up to $\mathscr{G}_1$ and $\mathscr{G}_2$. On the other hand, if $q_1\circ q_2[H]\in\mathscr{G}_0$ is a homotopy, then $H|_{e^0\times I}$ gives a self-homotopy of the basepoint. So, there is a one-to-one correspondence between $\mathscr{G}_0$ and $\ker(\eta)$. Moreover, it is easy to check the following properties by the homotopy extension theorem:
    \begin{enumerate}
    \item[(1)] If $s=0\in\textup{ker}(\eta)$, $H_s$ is homotopic to the trivial homotopy up to a homotopy fixing the basepoint;
    \item[(2)] Concatenating self-homotopies multiplies the subscripts: $[H_{st}]=[H_s]\cdot [H_t]$;
    \item[(3)] Reversing a self-homotopy inverts the subscript: $-[H_s]=[H_{s^{-1}}]$.
    \end{enumerate}
    Thus, there exists a group isomorphism $\mathscr{G}_0\cong\ker(\eta)$.
\end{proof}

\begin{definition}\label{def3.21}
    Define the affine action
    \begin{align*}\zeta:\mathscr{G}_0\times \mathbb{A}_{f_*} & \rightarrow\mathbb{A}_{f_*}\\
    (s,a) & \mapsto sas^{-1}+U_s
    \end{align*}
    where $s\in\mathscr{G}_0$, $U_s=\mu^*(q_1\circ q_2[H_s])\in\mathbb{A}_{f_*}$ and $H_s$ is an extension of $s$. Define $\mathbb{A}_f$ as the orbit space $\mathbb{A}_{f_*}/\mathscr{G}_0$.
\end{definition}

\begin{proposition}\label{prop3.25}
    $\zeta$ is a well-defined group action and $\mu^*:(\mathscr{H}^f/\mathscr{G}_2)/\mathscr{G}_1)\rightarrow\mathbb{A}_{f_*}$ induces a map $\overline\mu:((\mathscr{H}^f/\mathscr{G}_2)/\mathscr{G}_1)/\mathscr{G}_0\cong\mathscr{H}^f/\mathscr{H}^f_0\rightarrow\mathbb{A}_f$. And if $\overline{[H]}\in\mathscr{H}^f/\mathscr{H}^f_0$ satisfies $\overline\mu(\overline{[H]})=0$, then $\overline{[H]}=\overline{[E]}\in\mathscr{H}^f/\mathscr{H}^f_0$.
\end{proposition}

\begin{proof}
    Notice that $U_s\coloneqq\mu^*([H_s])$ is determined by $s$, and under concatenation, for $s,r\in\ker(\eta)$, $U_{s\cdot t}=U_s+sU_rs^{-1}$, this implies that the action $\zeta(s,a)\coloneq sas^{-1}+U_s$ satisfies
    $$\zeta(sr,a)=U_{sr}+srar^{-1}s^{-1}=\zeta(s,(U_r+rar^{-1}))=\zeta(s,\zeta(r,a)),$$ which is well-defined. Moreover, for $q_1\circ q_1[H]\in(\mathscr{H}^f/\mathscr{G}_2)/\mathscr{G}_1$ and $[H_s^*]\in\mathscr{G}_0$, we have $\mu^1(q_1\circ q_2[H_s\cdot H])=\mu^1(q_1\circ q_2[H_s])+s\mu^1(q_1\circ q_2[H])s^{-1}$. The conjugacy action appears because when we compute the intersection number of $[H]$ in $[H_s\cdot H]$, the basepoint moves along $s$, which coincides with the action of $\mathscr{G}_0$ on $\mathbb{A}_{f_*}$. Therefore, the induced map $\overline\mu:\mathscr{H}^f/\mathscr{H}^f_0\rightarrow\mathbb{A}_f$ is well-defined. And
    if $\overline\mu(\overline{[H]})=0\in\mathbb{A}_f$, then there exists $[H_s^*]\in\mathscr{G}_0$ such that $\zeta(s,\mu^1(q_1\circ q_1[H]))=\mu^1(q_1\circ q_2[H'])$, thus $\mu^1(q_1\circ q_2[H_s\cdot H])=\mu^1(q_1\circ q_2[H'])\in\mathbb{A}_{f_*}$. Thus, by Lemma \ref{lemma3.17}, \ref{lemma3.12} and Theorem \ref{thm2.3}, we learn that $\overline{[H]}=\overline{[E]}\in\mathscr{H}^f/\mathscr{H}^f_0$.
\end{proof}

\section{The invariant map and applications}
\label{sec_application}
With all the preparation, we can define the invariant map $\textup{fq}$.
\begin{definition}\label{def4.1}
     Define the invariant map $\textup{fq}$ as
\begin{align*}
    \textup{fq}:p^{-1}([f]) & \rightarrow\mathbb{A}_f\\
    f' & \mapsto \textup{fq}(f')\coloneq\overline\mu(\overline{[H]}),
\end{align*}

where $H$ is a homotopy from $f$ to $f'$.
\end{definition}
We claim that this map is a well-defined bijection, which proves Theorem \ref{thm1.1}.

\begin{proof}
To show $\textup{fq}$ is well defined, we need to show that it is independent of the choice of $H$ and is only determined by the isotopy class of $f'$, which we have discussed in detail in the last section. If there exist two homotopies $H$ and $H'$ from $f$ to $f'$, the product $[H]\cdot[-H']$ is a self-homotopy of $f$, so $0=\overline\mu(\overline{[H\cdot(-H')]})=\overline\mu(\overline{[H]})-\overline\mu(\overline{[H']})$, that is, $\overline\mu(\overline{[H]})=\overline\mu(\overline{[H']})$. Therefore, $\textup{fq}$ is well defined.

To show $\textup{fq}$ is a bijection, let $f_1,f_2\in p^{-1}([f])\cong\mathscr{H}^f/\mathscr{H}^f_0$ be two embeddings homotopic to $f$. By definition, $\textup{fq}(f_1)=\textup{fq}(f_2)$ implies that there exists $H_1$ connecting $f$ and $f_1$ and $H_2$ connecting $f$ and $f_2$, such that $\overline\mu(\overline{[H_1]})=\overline\mu(\overline{[H_2]})$. Thus, there exists $[G^*]\in\mathscr{G}_0$ $\overline\mu([G^*\cdot H_1])=\overline\mu([H_2])\in\mathbb{A}_{f_*}$. Thus consider the homotopy $H_3\coloneqq-H_2\cdot G\cdot H_1$, $H_3$ is a homotopy connecting $f$ and $f'$ such that
\begin{equation}
    \begin{split}
    \overline\mu([H_3]) &=\overline\mu([-H_2])+\overline\mu_3([H_s])+s\overline\mu([H_1])s^{-1}\\
    &=-\overline\mu([H])+U_s+s\overline\mu([H])s^{-1}\\
    &=-(U_s+s\overline\mu([H])s^{-1})+U_s+s\overline\mu_3([H])s^{-1}\\
    &=0\in\mathbb{A}_{f_*}
    \end{split}
\end{equation}

Then by Lemma \ref{lemma3.17}, \ref{lemma3.12} and Theorem \ref{thm2.3}, we learn that $[H_3]=[E]\in\mathscr{H}^f$, thus $f_1$ is isotopic to $f_2$, and so $\textup{fq}^i$ is injective. We finish the proof of Theorem \ref{thm1.1}.
\end{proof}

Now we get a bijection, and although we do not exactly know what the map $\phi$, $\psi$, and the action are. If we have some conditions on the homotopy groups of $N$, the obstructions can be removed. 

\begin{proof}[Proof of Corollary 1.3]
    In our bijection, if $N$ is simply-connected or if $f_*:\pi_1\Sigma\rightarrow\pi_1N$ is surjective, the quotient $\mathbb{A}_{[f]}$ has only a single element, which means that there is only 1 isotopy class. And if the surface admits an algebraic dual sphere $[S]\in\pi_3(N)$, which means $\lambda_N([S],f)=1$ and $\mu([S])=0\in\mathbb{A}_{[f]}$, then for any $\Sigma [g_i]\in\mathbb{Z}G\backslash\pi_3(N)/G$, we have 
    $$\phi(g\cdot [S])=0+[\lambda_N(\Sigma g_i\cdot [S],f)]=[\Sigma g_i\cdot\lambda_N([S],i)]=[\Sigma g_i]$$
    this means $\phi(\pi_3N)=\mathbb{A}_{[f]}$, and after the quotient the group $\mathbb{A}_{f_1}$ is trivial, and so is $\mathbb{A}_f$, which means there is only 1 isotopy class.
\end{proof}

Conversely,

\begin{proof}[Proof of Corollary 1.4]
    If $\pi_2N$, $\pi_3N$ are both trivial, then map $\phi$ and $\psi$ are both trivial, so $\mathbb{A}_f=\mathbb{A}_{[f]}/\ker\eta$. And since $\ker\eta\subset\pi_1N$ and if there are infinitely many conjugacy classes in $G\backslash\pi_1N/G$, there are infinitely many orbits, hence there are infinitely many isotopic classes.
\end{proof}

\bibliographystyle{amsalpha}
\bibliography{references}

\end{document}